\documentclass[12pt]{amsart}
\usepackage{amscd,amssymb,amsmath,amsthm,mathrsfs,lmodern,amsfonts,enumerate,stmaryrd}

\usepackage[all]{xy}
\usepackage{comment}

\usepackage[utf8]{inputenc}
\usepackage[T1]{fontenc}
\usepackage[colorlinks=true,citecolor=PineGreen,filecolor=PineGreen,linkcolor=Mahogany,pagebackref, hyperindex]{hyperref}
\usepackage[dvipsnames]{xcolor}
\usepackage{color, hyperref}

\newcommand{\p}{\mathfrak{p}}

\usepackage{fullpage, todonotes, amscd}

\usepackage{tikz}

\newcommand{\FF}{\mathbb{F}}
\newcommand{\NN}{\mathbb{N}}

\newcommand{\ds}{\displaystyle}
\newcommand{\m}{\mathfrak{m}}
\newcommand{\q}{\mathfrak{q}}
\renewcommand{\a}{\mathfrak{a}}
\renewcommand{\leq}{\leqslant}
\renewcommand{\geq}{\geqslant}

\theoremstyle{definition}
\newtheorem{theorem}{Theorem}[section]
\newtheorem{theoremx}{Theorem}

\numberwithin{equation}{section}

\newtheorem{question}[theorem]{Question}
\newtheorem{corollary}[theorem]{Corollary}
\newtheorem{lemma}[theorem]{Lemma}
\newtheorem{proposition}[theorem]{Proposition}

\theoremstyle{definition}
\newtheorem{definition}[theorem]{Definition}

\newtheorem{example}[theorem]{Example}

\newtheorem{remark}[theorem]{Remark}

\usepackage{amssymb,amsmath,amsthm}
\usepackage[english]{babel}
\usepackage{hyperref}

\newcommand{\hght}{\operatorname{ht}}

\newcommand{\mf}[1]{\mathfrak #1}

\newcommand{\Ass}{\operatorname{Ass}}

\DeclareMathOperator{\HH}{H}

\newcommand{\tpar}{\tau_{\rm par}}

\newcommand{\ps}[1]{\llbracket {#1} \rrbracket}
\title{Tight closure of products and F-rational singularities}

\author{Alessandro De Stefani}
\address{Dipartimento di Matematica, Universit{\`a} di Genova, Via Dodecaneso 35, 16146 Genova, Italy}
\email{alessandro.destefani@unige.it}

\author{Ilya Smirnov}
\address{BCAM -- Basque Center for Applied Mathematics, Mazarredo 14, 48009 Bilbao, Basque Country -- Spain}
\address{
Ikerbasque -- Basque Foundation for Science, Plaza Euskadi 5, 48009 Bilbao, Basque Country -- Spain
}
\email{ismirnov@bcamath.org}

\begin{document}
\begin{abstract} We prove a characterization of F-rationality in terms of tight closure of products of parameter ideals. Our results are inspired by the theory of complete ideals for surfaces and, in particular, the fundamental results of Lipman--Teissier and Cutkosky characterizing rational surface singularities in terms of products of complete ideals, 
but are valid also in higher dimensions. 
\end{abstract}
\maketitle

\section{Introduction}

The study of integrally closed\footnote{The original terminology was \emph{complete ideals}.} ideals in two-dimensional regular rings was initiated by Zariski \cite[Appendix~5]{ZariskiSamuel} and continued by Lipman \cite{Lipman}. Most notably, Lipman showed that integrally closed ideals on a rational surface singularity enjoy the same properties as integrally closed ideals in a regular local ring. 

\begin{theorem}{\cite[Theorem 7.1]{Lipman}}\label{Lipman} Let $(R,\m)$ be a normal two-dimensional local domain. If $R$ has rational singularities, then for any two integrally closed ideals $I,J$ one has $\overline{IJ}=IJ$.
\end{theorem}

Lipman's paper was very influential and led to many important developments. Among these, we find characterizations
of rational surface singularities by the properties of integrally closed ideals.

The definition of rational surface singularities is originally given by the vanishing of certain cohomology in a resolution of singularities. In order to address the absence of resolutions in positive characteristic, in \cite{LipmanTeissier}
Lipman and Teissier defined the class of pseudo-rational singularities. 
The main result of \cite{LipmanTeissier} is the Brian{\c c}on--Skoda theorem for pseudo-rational singularities and 
that, in fact, the Brian{\c c}on--Skoda 
theorem is a criterion for a surface singularity to be pseudo-rational.

\begin{theorem}{\cite[Proposition~5.5]{LipmanTeissier}}
If $(R, \m)$ be analytically unramified two-dimensional normal local ring. 
Then $R$ is pseudo-rational if and only if, for every $\m$-primary ideal $I$, one has $\overline{I^2} = I \overline{I}$. 
\end{theorem}


It should be noted that due to the existence of resolutions \cite{LipmanResolution}, pseudo-rationality and rationality coincide for surface singularities \cite[p.103, Example (a)]{LipmanTeissier}, so 
we get an ideal-theoretic characterization of rational surface singularities. 
This work was ``relativized'' by Okuma, Watanabe, and Yoshida in a series of papers started by \cite{OWY}. These papers define a geometric subclass of integrally closed ideals, called $p_g$-ideals, and show this class behaves rather similarly to integrally closed ideals on a rational singularity. Rational singularities are characterized by the property that every integrally closed ideal is $p_g$. 

In a different development, the Lipman--Teissier 
characterization was strengthened by Cutkosky,
providing a very satisfactory converse to Theorem~\ref{Lipman}.

\begin{theorem}{\cite[Theorem 1]{CutRat}} \label{thm Cutkosky}
Let $(R, \m)$ be an analytically normal local domain of dimension $2$ such that 
$R/\m$ is an algebraically closed field. The following are equivalent:
\begin{enumerate}
\item $R$ has rational singularity;
\item For any integrally closed $\m$-primary ideal $I$, one has $\overline{I^2} = I^2$;
\item For any two integrally closed ideals $I,J$, one has $\overline{IJ}=IJ$.
\end{enumerate}
\end{theorem}

The motivation of this paper was to search for a characteristic $p > 0$ version of this result, i.e., a theorem that would replace integral closure with tight closure in order to describe the class of F-rational singularities. This approach was also inspired by the tight closure approach to the Brian{\c c}on--Skoda theorem \cite{HoHu1,AberbachHuneke}.

Rational singularities in positive characteristic are not as well-behaved as in characteristic zero. Already in dimension two, there are surprising phenomena such as the failure of tautness (i.e., several non-isomorphic singularities with the same resolution graph) of rational double points in small characteristics \cite{Artin}. Moreover, since the resolution of singularities is generally not known, the usual definition of rational singularities is not available and it is not clear what should the right definition be. 

For many purposes, F-rational singularities can be regarded as the ``correct'' characteristic $p>0$ analogue of rational singularities in characteristic $0$.  It is known that F-rational singularities are normal, Cohen-Macaulay and pseudo-rational \cite{SmithFRational}. 
Our main result is the following.

\begin{theoremx}[see Theorem \ref{ThmRationalParIdeals}] \label{THMX}
Let $(R,\m)$ be an excellent F-injective normal domain. 
Then $R$ is F-rational if and only if, for every ideals $\q_1 \subseteq \q_2$ generated by systems of parameters, one has $(\q_1\q_2)^* = \q_1^*\q_2^*$.
\end{theoremx}

The statement of Theorem \ref{ThmRationalParIdeals} is actually more general: we do not assume that $R$ is F-injective to start with, and in place of that we add a condition in terms of Frobenius powers of parameter ideals.

After Theorem \ref{THMX}, and keeping in mind the analogies with the statements for rational singularities, a natural question is whether F-rationality can equivalent to every product of tightly closed ideals being tightly closed, at least in dimension two. However, after a personal communication with Craig Huneke, we were informed that this is not the case. Indeed, Brenner, Huneke, Mukundan and Verma were able to show that if $(R,\m)$ is F-rational and $(IJ)^* = I^*J^*$ for all $\m$-primary ideals $I,J$, then $R$ is weakly F-regular (see Theorem \ref{BHMV}). 

\subsection*{Acknowledgements} We thank Craig Huneke, Pham Hung Quy and Florian Enescu for helpful discussions, and Burt Totaro for providing feedback on an earlier version of this preprint. We also thank the referee for helping us improve the quality of the exposition. We are grateful to Holger Brenner, Craig Huneke, Vivek Mukundan and Jugal Verma for allowing us to include Theorem \ref{BHMV} in this paper.
The first named author was partially supported by the MIUR Excellence Department Project CUP D33C23001110001, and by INdAM-GNSAGA.
The second author was supported by Spanish Ministry of Science, Innovation, and  Universities under grants PCI2024-155055-2 funded by MICIU/AEI/10.13039/501100011033 and UE,
RYC2020-028976-I funded by MICIU/AEI/10.13039/501100011033 and by FSE ``invest in your future'', and
EUR2023-143443 funded by MICIU/AEI/10.13039/501100011033 and the European Union NextGenerationEU/PRTR.

\section{Background}

Throughout this article, $(R,\m)$ denotes a local ring of prime characteristic $p > 0$. 
 Given a proper ideal $I \subseteq R$, and $q=p^e$, we let  $I^{[q]} = (y^q \mid y \in I)$ be its Frobenius power. Note that $I^{[q]} = F^e(I)R$, where $F\colon R \to R$ is the Frobenius map, that is, the ring homomorphism raising an element to its $p$-th power. We will use $\ell_R(-)$ to denote the length of an $R$-module.

\begin{definition}
An element $x \in R$ belongs to the Frobenius closure $I^F$ of $I$ if $x^q \in I^{[q]}$ for all $q=p^e \gg 0$.
\end{definition}

 We denote by $R^\circ$ the complement of the minimal primes of $R$.
 
 \begin{definition}
 An element $x \in R$ belongs to the tight closure $I^*$ of $I$ if there exists $c \in R^\circ$ such that $cx^q \in I^{[q]}$ for all $q=p^e \gg 0$.  
 \end{definition}
Clearly, we have $I \subseteq I^F \subseteq I^*$. Moreover, $I^* \subseteq \overline{I} \subseteq \sqrt{I}$.
\begin{definition} Let $(R,\m)$ be a local ring.
\begin{itemize}
\item A parameter is an element $x \in \m$ such that $\dim(R/(x)) = \dim(R)-1$. A system of parameters is a sequence $x_1,\ldots,x_t$ such that $x_{i+1}$ is a parameter in $R/(x_1,\ldots,x_i)$ for all $i=0,\ldots,t-1$. 
\item A parameter ideal $\q$ is an ideal generated by a system of parameters $x_1,\ldots,x_t$ for some $0 \leq t \leq d$, where the case $t=0$ corresponds to the choice $\q=(0)$.
\item A filter regular element is $x \in \m$ such that $0:_R x$ has finite length. A filter regular sequence are elements $x_1,\ldots,x_t$ such that $x_{i+1}$ is a filter regular element in $R/(x_1,\ldots,x_i)$ for all $i=0,\ldots,t-1$.
\end{itemize} 
\end{definition}

\begin{remark} \label{rmk filter regular} An element $x \in \m$ is filter regular if and only if $x \notin \p$ for all $\p \in \Ass(R) \smallsetminus \{\m\}$. In particular, a filter regular element is a parameter as long as $d>0$. In particular, by prime avoidance, we can always find a full system of parameters $x_1,\ldots,x_d$ which is a filter regular sequence.
\end{remark}

\begin{definition}
We say that $(R,\m)$ is F-rational if $\q = \q^*$ for every parameter ideal $\q \subseteq R$. 
\end{definition}

Recall that an F-rational local ring $R$ is always normal, and it is Cohen-Macaulay if it is the homomorphic image of a Cohen-Macaulay ring \cite[Theorem (4.2)]{HoHu2}. This happens, for instance, if $R$ is excellent (see \cite{Kawasaki}). It was proved by Smith that a $d$-dimensional Cohen-Macaulay local ring is F-rational if and only if its top local cohomology module $H^d_\m(R)$ has no proper submodules which are stable under the action of the induced Frobenius map $F:H^d_\m(R) \to H^d_\m(R)$ \cite{SmithFRational}.

\begin{definition} The {\it parameter test ideal of $R$} is 
\[
\ds \tpar(R) = \bigcap_\q (\q: \q^*),
\]
where $\q$ runs over all parameter ideals (equivalently, all ideals generated by a full system of parameters).
\end{definition}

Observe that $R$ is F-rational if and only if $\tpar(R) = R$. Moreover, if $R$ is equidimensional, then it is F-rational on the punctured spectrum if and only if $\tpar(R)$ is an $\m$-primary ideal.

\begin{definition}
We say that $(R,\m)$ is F-injective if the map induced on $\HH^i_\m(R)$ by the Frobenius map on $R$ is injective for all $i \in \NN$.
\end{definition}

It is well-known that F-rational rings are F-injective. F-injectivity is related to the Frobenius closure of parameter ideals: if every parameter ideal is Frobenius closed then $R$ is F-injective \cite[Main Theorem A]{PhamShimomoto}. The converse does not hold in general, see \cite[Main Theorem B]{PhamShimomoto}, but it does if $\HH^i_\m(R)$ has finite length for each $i \ne d$, as shown in \cite[Theorem 1.1]{Ma}. Note that the condition that $\HH^i_\m(R)$ has finite length for each $i \ne d$ is equivalent to $R$ being Cohen-Macaulay on the punctured spectrum when $R$ is equidimensional.

\begin{lemma} \label{LemmaTightFrobenius} Let $(R,\m)$ be a local ring, and assume that $\tpar(R)$ contains an ideal $I$. For every parameter ideal $\q$ of $R$, $\q^*/\q^F$ is annihilated by $\sqrt{I}$. In particular, if $R$ is equidimensional and F-rational on the punctured spectrum, then $\q^*/\q^F$ is an $R/\m$-vector space.
\end{lemma}
\begin{proof}
The proof of the first claim is analogous to \cite[Proposition 8.13]{HoHu3}: we let $J = \sqrt{I}$ and 
may assume that $J^{[q]} \subseteq \tpar(R)$ for some $q=p^{e}$. Then 
\[
(J\q^*)^{[q]} \subseteq J^{[q]} (\q^{[q]})^* \subseteq \tpar(R) (\q^{[q]})^* \subseteq \q^{[q]}.
\]
It follows that $J\q^* \subseteq \q^F$. The second claim follows immediately from the first and from previous observations.
\end{proof}

\begin{corollary}\label{cor: tight closure is vector space}
Let $(R,\m)$ be an F-injective equidimensional local ring which is F-rational on the punctured spectrum.
Then any parameter ideal $\q$ is Frobenius closed and  
$\m \q^* \subseteq \q$. 
\end{corollary}
\begin{proof}
Because $R$ is F-rational on the punctured spectrum and equidimensional, it follows that $\HH^i_\m(R)$ has finite length for all $i \ne \dim(R)$. Thus $\q^F = \q$  by \cite[Theorem~1.1]{Ma} and the assertion follows from Lemma~\ref{LemmaTightFrobenius}.
\end{proof}

\section{Tight closure of products of parameter ideals}

We start with a condition on Frobenius powers of parameter ideals which guarantees that a ring is F-rational.

\begin{proposition} \label{PropFInjectiveFrobenius}
Let $(R,\m)$ be a local ring. If every parameter ideal $\q$ satisfies $(\q^F)^{[p]} = (\q^{[p]})^F$, then $R$ is F-injective. 
\end{proposition}
\begin{proof}
First, note that choosing $\q=(0)$ we get that $((0)^F)^{[p]} = (0)^F$, and thus $(0)^F=(0)$ by Nakayama's Lemma. It follows that $R$ is reduced. Let $d=\dim(R)$, and $x_1,\ldots,x_d$ be a filter regular sequence, see Remark \ref{rmk filter regular}. For $1 \leq t \leq d$ set $\q=(x_1,\ldots,x_t)$, and observe that $(\q^F)^{[q]}=\q^{[q]}$ for all $q =p^e \gg 0$. Since $\q^{[q/p]}$ is also a parameter ideal, we have 
\[
\ds (\q^{[q]})^F = ((\q^{[q/p]})^{[p]})^F = ((\q^{[q/p]})^F)^{[p]} = (\q^F)^{[q]}.
\]
Thus there exists $q_0 = q_0(\q)$ such that for $q \geq q_0$ we have $(\q^{[q]})^F = \q^{[q]}$. In particular, replacing $\q$ with $\a:=\q^{[q_0]}$, we have that $\a$ is generated by a filter regular sequence, and $\a^{[q]}$ is Frobenius closed for all $q$. It follows from \cite[Theorem 3.7]{PhamShimomoto} that the Frobenius action on $\HH^t_\m(R)$ is F-injective. Since this holds for all $1 \leq t \leq d$, and $\HH^0_\m(R)=0$ because $R$ is reduced, we conclude that $R$ is F-injective.
\end{proof}

We recall the definition of special tight closure, see \cite{Vraciu,HunekeVraciu,Epstein} for more details on this notion.

\begin{definition} Let $(R,\m)$ be a local ring, and $I \subset R$ be an ideal. The special part of the tight closure of $I$ is defined as follows:
\[
I^{*sp} = \{x \in R \mid  x^{q_1} \in (\mf m I^{[q_1]})^* \text{ for some } q_1=p^{e_1}\}.
\]
\end{definition}

\begin{theorem}\label{THMFRationalforFInjective}
Let $(R, \m)$ be an $F$-injective excellent normal domain which is $F$-rational on the punctured spectrum. Then $R$ is $F$-rational if and only if for some (equivalently, for all) $\m$-primary parameter ideal $\q$, we have $(\q \q^{[q]})^* = \q^* (\q^{[q]})^*$ for all $q=p^e$.
\end{theorem}
\begin{proof}
Assume that $R$ is F-rational, and let $\q = (x_1,\ldots,x_d)$ be a parameter ideal, so that $(\q^{[q]})^* = \q^{[q]}$ for all $q=p^e$. The containment $\q^*(\q^{[q]})^* = \q\q^{[q]} \subseteq (\q\q^{[q]})^*$ is clear. For the other containment, observe that $\q\q^{[q]}$ is generated by monomials in $x_1,\ldots,x_d$. By \cite{EagonHochster} we have that $\q\q^{[q]} = J_1 \cap \ldots \cap J_t$, where each $J_i$ is a parameter ideal generated by monomials in the regular sequence $x_1,\ldots,x_d$. Then
\[
(\q\q^{[q]})^* = (J_1 \cap \ldots \cap J_t)^* \subseteq J_1^* \cap \ldots \cap J_t^* = J_1 \cap \ldots \cap J_t = \q\q^{[q]}.
\]
Conversely, let $\q$ be a parameter ideal such that $(\q\q^{[q]})^* = \q^* (\q^{[q]})^*$ for all $q=p^e$. By \cite[Theorem~4.5]{Epstein}, for some $q_0$ we have a containment $(\q^*)^{[q_0]} \subseteq \q^{[q_0]} + (\q^{[q_0]})^{*sp}$. 

We now claim that $(\q^{[q_0]})^{*sp} = \q^{[q_0]}$. Take $x \in (\q^{[q_0]})^{*sp}$, then $x^{q_1} \in (\m \q^{[q_0q_1]})^*$ for some $q_1=p^{e_1}$. 
Let $q_2 \gg 0$ be such that $\m^{[q_2]} \subseteq \q$.
Then we have
\begin{align*}
x^{q_1q_2} & \in ((\m \q^{[q_0q_1]})^*)^{[q_2]} \subseteq (\m^{[q_2]} \q^{[q_0q_1q_2]})^*  \subseteq (\q \q^{[q_0q_1q_2]})^* \\
& = \q^* (\q^{[q_0q_1q_2]})^* \subseteq  \m (\q^{[q_0q_1q_2]})^* \subseteq \q^{[q_0q_1q_2]},
\end{align*}
where the last containment holds by Corollary~\ref{cor: tight closure is vector space}.
This shows that $x \in (\q^{[q_0]})^F = \q^{[q_0]}$ again by Corollary~\ref{cor: tight closure is vector space}. Therefore, $(\q^*)^{[q_0]} \subseteq \q^{[q_0]} + (\q^{[q_0]})^{*sp} = \q^{[q_0]}$, 
so that $\q^* = \q$ by Corollary~\ref{cor: tight closure is vector space}. 
This proves that $R$ is F-rational by \cite[Theorem~6.27]{HoHu2}. 
\end{proof}

\begin{remark}
It should be noted that normality in Theorem \ref{THMFRationalforFInjective} is not essential, \cite[Theorem~4.5]{Epstein} only requires that $R$ is analytically irreducible and that the residue fields of $R$ and its normalization coincide. The normality assumption in Theorem~\ref{THMFRationalforFInjective} is added to ensure both conditions; it should be noted that an F-rational ring is normal by \cite[Theorem~4.2]{HoHu2}.
\end{remark}

The following example shows that analytic irreducibility is crucial in Theorem \ref{THMFRationalforFInjective}.

\begin{example} \label{ExOneDimensional} Let $k$ be an algebraically closed field of characteristic $p > 0$, and $R=k \ps{x,y}/(xy)$. A parameter ideal $\q$ of $R$ can be generated by an element of the form $x^n+uy^m$, where $u \in R$ is a unit. Moreover in this case $\q^* = \overline{\q} = (x^n,y^m)$, so that
\[
(\q\q^{[q]})^* = (x^{n(q+1)}+u^{q+1}y^{m(q+1)})^* = (x^{n(q+1)},y^{m(q+1)}) = (x^n,y^m)(x^{nq},y^{mq})= \q^*(\q^{[q]})^*.
\]
However, $R$ is not F-rational because it is not normal, see \cite[Theorem~4.2]{HoHu2}.
\end{example}

Using an independence theorem of Pham Hung Quy \cite{Quy_parameter}, we show that under suitable assumptions tight closure commutes with Frobenius powers for parameter ideals.

\begin{theorem} \label{thm F-rational perfect} Let $(R,\m)$ be an equidimensional excellent F-injective local ring that is F-rational on the punctured spectrum. If $R/\m$ is perfect, then $(\q^*)^{[q]} = (\q^{[q]})^*$ for every parameter ideal $\q$ and all $q=p^e$.
\end{theorem}
\begin{proof}
If $R$ is F-rational there is nothing to prove, so assume that $R$ is not F-rational. By Corollary~\ref{cor: tight closure is vector space} we have that $\q^*/\q$ is an $R/\m$-vector space. Let $u_1,\ldots,u_t \in \q^*$ be a basis modulo $\q$. By the Main Theorem of \cite{Quy_parameter}, we have that $\ell_R((\q^{[q]})^*/\q^{[q]})$ does not depend on $q$. 
Since $u_1^q,\ldots,u_t^q$ belong to $(\q^*)^{[q]}\subseteq (\q^{[q]})^*$, and because $(\q^{[q]})^*/\q^{[q]}$ is an $R/\m$-vector space of dimension $t = \ell_R(\q^*/\q)$, to complete the proof it suffices to show that $u_1^q,\ldots,u_t^q$ are linearly independent modulo $\q^{[q]}$. To this end, assume that $\lambda_1 u_1^q + \ldots + \lambda_t u_t^q \in \q^{[q]}$. We want to show that $\lambda_i \in \m$ for all $i$. Since the residue field is perfect, we can choose $\alpha_1,\ldots,\alpha_t \in R$ such that $\alpha_j^q\equiv  \lambda_j$ modulo $\m$. By Corollary~\ref{cor: tight closure is vector space} we have that $\m u_j^q \subseteq \m(\q^{[q]})^* \subseteq \q^{[q]}$ for all $j$, and therefore we get that $(\alpha_1 u_1 + \ldots + \alpha_t u_t)^q \equiv \lambda_1 u_1^q + \ldots + \lambda_t u_t^q  \equiv 0$ modulo $\q^{[q]}$, that is, $\alpha_1 u_1 + \ldots + \alpha_t u_t \in \q^F$. 
By Corollary \ref{cor: tight closure is vector space} we conclude that $\alpha_1 u_1 + \ldots + \alpha_t u_t \in \q$, and it follows from our choice of $u_1,\ldots,u_t$ that $\alpha_1,\ldots,\alpha_t \in \m$. Finally, this implies that $\lambda_1,\ldots,\lambda_t \in \m$, and finishes the proof.
\end{proof}

We point out that Theorem \ref{thm F-rational perfect} is not true if we do not assume F-injective, even for two-dimensional hypersurfaces with minimal multiplicity.

\begin{example}
Let $R=\FF_2\ps{x,y,z}/(x^2 + y^3 + z^5)$. Consider $\q=(y,z)$, that is a minimal reduction of $\m = (x,y,z)$. Then $R$ is not F-injective, as $x \in (y,z)^F$. Because $R$ is Gorenstein and not F-rational, we get $\q^* \supseteq \q:\m =  \m$, so that $(\q^*)^{[2]} = \m^{[2]} = \q^{[2]}$. Finally, again because $R$ is not F-rational, it follows that $(\q^{[2]})^* \supsetneq \q^{[2]}$.
\end{example}

We now want to remove the assumption that $R$ is F-rational on the punctured spectrum from Theorem \ref{THMFRationalforFInjective}. First we prove an auxiliary result.

\begin{proposition}\label{localization}
Let $R$ be a locally equidimensional Noetherian ring of characteristic $p > 0$ that is a homomorphic 
image of a Cohen-Macaulay ring. Let $x_1, \ldots, x_t$ permutable parameters in $R$ and 
$J$ be an ideal generated by monomials in $x_1,\ldots,x_t$. 
If $R$ has a test element, then tight closure commutes with localization for $J$.
\end{proposition}
\begin{proof}
Let $S$ be a multiplicatively closed set. In our assumptions, tight closure commutes with localization at $S$ for any parameter ideal by \cite[Theorem~5.4, Theorem~5.21]{AHH}. In order to apply \cite[Theorem~5.21]{AHH} note that if $\q$ is a parameter ideal, then $\hght(\q R/P) = \hght(\q)$ for all minimal primes $P$ of $R$ because $R$ is locally equidimensional.

Now let $A=\FF_p[X_1,\ldots,X_t]$ be a polynomial ring, and consider the ring map $A \to R$ sending $X_i \mapsto x_i$. Let $U_1,\ldots,U_s$ be monomials in $A$ such that $J = (U_1,\ldots,U_s)R$. Let $I_1,\ldots,I_r$ be irreducible monomial ideals of $A$ such that $I = I_1 \cap \ldots \cap I_r$, and set $J_i = I_iR$ for $i=1,\ldots,r$. One always has $J \subseteq \bigcap_{i=1}^r J_i$, and equality holds if $x_1,\ldots,x_t$ forms a regular sequence \cite{EagonHochster}. On the other hand, by \cite[Theorem~7.9]{HoHu1} there exists $c \in R^\circ$ such that $c\left(\bigcap_{i=1}^r J_i\right)^{[q]} \subseteq \left(\left(\bigcap_{i=1}^r I_i\right)R\right)^{[q]} = J^{[q]}$ for all $q \gg 0$. By \cite[Theorem~7.12]{HoHu1}, using that $A$ is regular, we get that $\bigcap_{i=1}^r J_i^* \subseteq J^*$ (see also \cite[Corollary 1.3]{Quy_parameter_decomp} for the case in which $J$ is the power of a parameter ideal). Since each $J_i$ is generated by a system of parameters, we finally get
\begin{align*}
(S^{-1}J)^* & \subseteq \left (S^{-1}\bigcap_{i = 1}^r J_i \right)^* = \left (\bigcap_{i = 1}^r S^{-1}J_i \right)^*  \\ & \subseteq \bigcap_{i = 1}^r (S^{-1}J_i)^* 
= \bigcap_{i = 1}^r S^{-1}(J_i)^* \\ & = S^{-1} \left ( \bigcap_{i = 1}^r J_i^* \right)  \subseteq S^{-1}J^* \subseteq (S^{-1}J)^*. \qedhere
\end{align*}
\end{proof}

\begin{theorem} \label{ThmRationalParIdeals} Let $(R,\m)$ be an excellent normal domain. 
The following conditions are equivalent:
\begin{enumerate}
\item $R$ is F-rational;
\item $(\q_1^{[p]})^F = (\q_1^F)^{[p]}$ and $(\q_1\q_2)^* = \q_1^*\q_2^*$ for all ideals $\q_1 \subseteq \q_2$ generated by a system of parameters;
\item $R$ is F-injective and $(\q_1\q_2)^* = \q_1^*\q_2^*$ for all ideals $\q_1 \subseteq \q_2$ generated by a system of parameters;
\item $R$ is F-injective and $(\q\q^{[q]})^* = \q^*(\q^{[q]})^*$ for all ideals $\q$ generated by a system of parameters, and all $q=p^e$.
\end{enumerate}
\end{theorem}
\begin{proof}
If $R$ is F-rational, then $(\q_1^{[p]})^F \subseteq (\q_1^{[p]})^* = \q_1^{[p]} \subseteq (\q_1^F)^{[p]} \subseteq (\q_1^{[p]})^F$ shows the first equality of (2). For the second claim, let $\alpha \in (\q_1\q_2)^*$, and observe that $\alpha \in (\q_1 \cap \q_2)^* \subseteq \q_1^* \cap \q_2^*  = \q_1 \cap \q_2 = \q_1$. If we let $\q_1=(x_1,\ldots,x_t)$, then we can write $\alpha = \sum_{i=1}^t a_i x_i$. Moreover, by assumption there exists $c \ne 0$ such that $c\alpha^q \in (\q_1\q_2)^{[q]}$ for all $q \gg 0$. That is, we can write $\sum_{i=1}^t ca_i^qx_i^q = \sum_{i=1}^t r_ix_i^q$ for some $r_i \in \q_2^{[q]}$. Since $R$ is F-rational, it is Cohen-Macaulay, and therefore the elements $x_1^q,\ldots,x_t^q$ form a regular sequence. From above we conclude that $ca_i^q \in \q_2^{[q]} + (x_1,\ldots,\widehat{x_i},\ldots,x_t)^{[q]} \subseteq \q_2^{[q]}$, where we used that $x_j^q \in \q_1^{[q]} \subseteq \q_2^{[q]}$ for all $j=1,\ldots,t$. As this holds for all $q \gg 0$, we conclude that $a_i \in \q_2^* = \q_2$ for all $i=1,\ldots,d$, and thus $\alpha \in \q_1\q_2$. 

(2) implies (3) by Proposition \ref{PropFInjectiveFrobenius}, and (3) clearly implies (4). Now we assume (4). By way of contradiction we assume that $R$ is not F-rational. By \cite[Theorem~3.5]{Velez} the F-rational locus of $R$ is open. We fix $\p$ 
to be a minimal prime of the non F-rational locus, so that $R_\p$ is F-rational on the punctured spectrum but not F-rational itself. 
We can find a partial system of parameters $x_1,\ldots,x_t$ inside $\p$, where $t={\rm ht}(\p)$. Let $\q$ be the ideal generated by these elements. By assumption we have that $(\q\q^{[q]})^* = \q^*(\q^{[q]})^*$ for all $q=p^e$. Because an excellent local ring is a homomorphic image of a Cohen-Macaulay ring (\cite{Kawasaki}), by Proposition~\ref{localization} this condition is preserved after localizing at $\p$. 
Moreover, since $R$ is F-injective, $R_\p$ is also F-injective by \cite[Proposition~3.3]{DattaMurayama}.

We now reduced to the case in which $R_\p$ is an F-injective ring which is F-rational on the punctured spectrum and the parameter ideal $\a=\q R_\p$ in $R_\p$ and is such that $(\a\a^{[q]})^* = \a^*(\a^{[q]})^*$ for all $q=p^e$. Theorem~\ref{THMFRationalforFInjective} implies that $R_\p$ is F-rational, a contradiction.
\end{proof}

In trying to find analogies between Theorem~\ref{ThmRationalParIdeals} and \cite[Theorem 1]{CutRat}, one might wonder whether the condition $(IJ)^*= I^*J^*$ \emph{for all} ideals $I,J \subseteq R$ may be equivalent to $R$ being F-rational. This is actually not the case: the following argument was shown to us in a private communication by Brenner, Huneke, Mukundan, and Verma. 
\begin{theorem} \label{BHMV}
Let $(R,\m)$ be an F-rational local ring of dimension $d$. The following are equivalent:
\begin{enumerate}
\item $R$ is weakly F-regular,
\item $(IJ)^*=I^*J^*$ for all ideals $I,J \subseteq R$,
\item $(IJ)^* = I^*J$ for all ideals $I$ and all parameter ideals $J$.
\end{enumerate}
\end{theorem}
\begin{proof}
The implications (1) $\Rightarrow$ (2) $\Rightarrow$ (3) are straightforward. Assume (3). In order to show that $R$ is weakly F-regular it suffices to prove that every $\m$-primary ideal is tightly closed. Let $I$ be an $\m$-primary ideal, we induct on $n = \mu(I) \geq d = \dim(R)$. If $n=d$ then $I$ is a parameter ideal, and hence it is tightly closed since $R$ is F-rational. Assume $n>d$. We claim that, we can find minimal generators $x_1,\ldots,x_n$ of $I$ such that $J=(x_1,\ldots,x_d)$ is a parameter ideal, and $K=(x_1,\ldots,x_{d-1},x_{d+1},\ldots,x_n)$ is still $\m$-primary. In fact, choose any parameter ideal $J=(x_1,\ldots,x_d)$ inside $I$. Observe that $I \not\subseteq J+\m I$, othwerwise we would have $I=J$ by Nakayama's lemma, and thus $n=d$. Moreover, for any minimal prime $\p$ of $(x_1,\ldots,x_{d-1})$, we have $I \not\subseteq \p$, since otherwise $\hght(I) \leq \hght(\p) =d-1$. It follows by prime avoidance that we can find $x_{d+1} \in I \smallsetminus (J +\m I)$ such that $x_{d+1} \notin \p$ for all minimal primes $\p$ of $(x_1,\ldots,x_{d-1})$. The latter condition ensures that $(x_1,\ldots,x_{d-1},x_{d+1})$ is $\m$-primary. Moreover, our choice of $x_{d+1}$ guarantees that the residue class $\overline{x_{d+1}} \in I/\m I$ does not belong to the $R/\m$-vector subspace spanned by $\overline{x_1},\ldots,\overline{x_d}$. By Nakayama's lemma this implies that $x_1,\ldots,x_d,x_{d+1}$ are part of a minimal generating set for $I$. Now complete $\overline{x_1},\ldots,\overline{x_d},\overline{x_{d+1}}$ to a $R/\m$-basis of $I/\m I$ with elements $\overline{x_{d+2}},\ldots,\overline{x_n}$, and the elements $x_1,\ldots,x_d,x_{d+1},\ldots,x_n$ satisfy the desired properties.

Now note that $\mu(K)<n$, so by induction we have that $K^*=K$. Let $z \in I^*$. There exists $c \in R^\circ$ such that $cz^q \in I^{[q]}$ for all $q \gg 0$. Multiplying by $x_1^q$ we get that $cx_1^qz^q \in x_1^qI^{[q]} \subseteq J^{[q]}K^{[q]}$. It follows that $x_1z \in (JK)^* = JK^* = JK$. Write $x_1z = \sum_{i=1}^d a_i x_i$ with $a_i \in K \subseteq I$, so that $(z-a_1)x_1 \in (x_2,\ldots,x_d)$. Since $x_1,\ldots,x_d$ is a regular sequence, we conclude that $z \in (a_1,x_2,\ldots,x_d) \subseteq I$, as desired.
\end{proof}

\begin{corollary} Let $(R,\m)$ be an excellent F-injective normal local ring. Then $R$ is weakly F-regular if and only if $(IJ)^* = I^*J^*$ for all $\m$-primary ideals $I,J \subseteq R$.
\end{corollary}
\begin{proof}
The proof is a straightforward application of Theorem \ref{ThmRationalParIdeals} and Theorem \ref{BHMV}. 
\end{proof}

\section{Ending remarks and questions}

In light of Theorem \ref{ThmRationalParIdeals}, Pham Hung Quy asked to us the following question.
\begin{question} \label{Quy} If $(R,\m)$ is F-rational, and $\q_1,\q_2$ are parameter ideals, is $\q_1\q_2$ tightly closed? 
\end{question}
With the aid of Macaulay 2 \cite{Mac2} we are able to give a negative answer to Question \ref{Quy}.

\begin{example} Let $K=\FF_2(u,v)$ and $S=K[x,y,z]/(x^2+uy^2+vz^2)$, with the standard grading. Finally, let $R=S^{(2)}$, the second Veronese subring of $R$. By \cite{GSS}, $R$ is F-rational but not F-pure. It can be checked using Macaulay 2 that $\q_1=(x^2,y^2)R$ and $\q_2=(xy,z^2)R$ are parameter ideals in $R$, and that $yz^3 \in R \smallsetminus \q_1\q_2$. However, 
\begin{align*}
(yz^3)^2 & = y^2z^4(v^{-1}x^2+uv^{-1}y^2) = v^{-1}x^2y^2(v^{-2}x^4+u^2v^{-2}y^4) + uv^{-1}y^4z^4 \\
& = v^{-3}(x^3y)^2 + u^2v^{-3}(xy^3)^2 + uv^{-1}(y^2z^2)^2 \in (\q_1\q_2)^{[2]},
\end{align*}
and it follows that $\q_1\q_2$ is not even Frobenius closed.
\end{example}

As already mentioned, Cutkosky proved that for an analytically normal surface having rational singularities is equivalent to the square of every integrally closed $\m$-primary ideal being integrally closed \cite[Theorem 1]{CutRat}. While Theorem \ref{BHMV} shows that products of integrally closed and tightly closed ideals have substantially different behavior in terms of characterizing rational and F-rational surface singularities, one might still hope that powers behave more similarly. We do not know at the moment whether this is true or not:
\begin{question} \label{Q1} Let $(R,\m)$ be a two-dimensional excellent F-injective analytically normal local ring. Is being F-rational equivalent to the condition $(I^2)^* = (I^*)^2$ for all $\m$-primary ideals $I \subseteq R$? 
\end{question}

\begin{remark}
We point out that the condition that $(\q^2)^* = (\q^*)^2$ for all ideals $\q \subseteq R$ generated by a full system of parameters does not imply that $R$ is F-rational. For example, take $R$ to be a two-dimensional 
Gorenstein pseudo-rational F-injective  local ring. We point out that such a ring needs not be F-rational; for instance, take $R=\FF_2\ps{x,y,z}/(z^2+x^2y+xy^2+xyz)$. Let $\q=(f,g)$ be a parameter ideal of $R$.
Since the test ideal $\tau(R)$ equals $\m$ we have $I^* \subseteq I:\m$ 
for any $I$. Furthermore, one always has $(\q^*)^2 \subseteq (\q^2)^*$, so
\[
(\q: \m)^2 = (\q^*)^2 \subseteq (\q^2)^* \subseteq \q^2 : \m,
\]
Now, if we let $\q:\m=(f,g,u)$, then observe that $(f^2,g):\m=(f^2,g,u f)$ and $(f,g^2):\m = (f,g^2,u g)$. From the injective map $0 \to R/\q^2 \to R/(f^2,g) \oplus R/(f,g^2)$, one gets that the socle of $R/\q^2$ is contained in the socle of $R/(f^2,g) \oplus R/(f,g^2)$, and this implies that $\q^2:\m \subseteq (\q^2,u f,u g)$. Thus $\q^2:\m \subseteq (\q:\m)^2$, and we have equality $(\q^2)^* = (\q^*)^2$.

On a positive side, a similar argument shows that if $R$ is a Cohen-Macaulay then $\q^2$ is tightly closed if and only if $\q$ is tightly closed, i.e., $R$ is F-rational.
\end{remark}

We now make another remark: combining the Theorem of Huneke--Itoh \cite{HunekeStable,Itoh} with the aforementioned result of Lipman--Teissier on the Brian{\c c}on--Skoda Theorem, we know that a two-dimensional normal local ring  is pseudo-rational if and only if $\overline{\q^2} \subseteq \q$ for every ideal $\q$ generated by a system of parameters. In light of this equivalence, and keeping in mind the tight closure approach to the Brian{\c c}on--Skoda Theorem \cite{HoHu1,AberbachHuneke}, another natural question arises: is being F-rational equivalent to $(\q^2)^* \subseteq \q$ for every ideal $\q$ generated by a system of parameters? However, this is false because one can take any pseudo-rational but not F-rational surface singularity, and exploit the containment $(\q^2)^* \subseteq \overline{\q^2}$. For a concrete example, take $R=\FF_2\ps{x,y,z}/(z^2+x^2y+xy^2+xyz)$ as above. 

\begin{question} Smith proved that F-rational excellent local rings are pseudo-rational \cite{SmithFRational}. Is it possible to use Theorem \ref{ThmRationalParIdeals} to find an alternative proof of this result? The two-dimensional case might be more approachable in view of the various characterizations we have recalled in terms of complete ideals.
\end{question}

\bibliographystyle{alpha}
\bibliography{References}

\end{document}